\documentclass[12pt, twoside, a4paper, fleqn]{article}

\usepackage{latexsym}
\usepackage{amscd}
\usepackage{theorem}
\usepackage{pifont}
\usepackage{mathbbol}
\usepackage{amsfonts}
\usepackage{xspace}
\usepackage{amssymb}
\usepackage{fancyhdr}
\usepackage{mathrsfs}
\usepackage{amsmath}%
\usepackage{graphics}%
\usepackage{graphicx}%
\usepackage{euscript}%
\usepackage{psfrag}%
\usepackage{empheq}%
\usepackage{upgreek}
\usepackage{eufrak}

\DeclareMathAlphabet{\mathpzc}{OT1}{pzc}{m}{it}

{\theorembodyfont{\slshape} \newtheorem{theorem}{Theorem}[section]}
{\theorembodyfont{\slshape} }
{\theorembodyfont{\slshape} }
{\theorembodyfont{\slshape} \newtheorem{corollary}[theorem]{Corollary}}

\newenvironment{proof}{\noindent\textbf{Proof:\ }}{$\hfill{\bullet}$}

\numberwithin{equation}{section}

\bigskip

\title{The dependence of Lyapunov exponents of polynomials on its coefficients}
\bigskip
\bigskip

\author{{\sc Shrihari Sridharan\footnote{This author was supported by a Fasttrack Grant for Young Scientists awarded by the Department of Science and Technology, Government of India, vide SR/FTP/MS-008/2012.},\ \ \ \ \ \ Atma Ram Tiwari} \\ Indian Institute of Science Education and Research \\ Thiruvananthapuram (IISER-TVM) \\ {\tt shrihari@iisertvm.ac.in},\ \ \ \ \ \ {\tt artiwari15@iisertvm.ac.in}}
\bigskip

\date{\today}
\bigskip 

\begin{document}

\maketitle

\begin{abstract}
\noindent 
In this paper, we consider the family of hyperbolic quadratic polynomials parametrised by a complex constant; namely $P_{c} (z) = z^{2} + c$ with $|c| < 1$ and the family of hyperbolic cubic polynomials parametrised by two complex constants; namely $P_{(a_{1}, a_{0})} (z) = z^{3} + a_{1} z + a_{0}$ with $|a_{i}| < 1$, restricted on their respective Julia sets. We compute the Lyapunov characteristic exponents for  these polynomial maps over corresponding Julia sets, with respect to various Bernoulli measures and obtain results pertaining to the dependence of the behaviour of these  exponents on the parameters describing the polynomial map. We achieve this using the theory of thermodynamic formalism, the pressure function in particular. 
\end{abstract}
\bigskip

\begin{tabular}{l c l}
\textbf{Keywords} & : & Lyapunov exponents, \\ & & Structural stability of systems, \\ & & Hausdorff dimension. \\ \\
\textbf{AMS Subject Classifications} & : & 37B25, 37F10, 37F15.
\end{tabular}
\bigskip  

\thispagestyle{empty}

\section{Introduction}
The study of Lyapunov  exponents plays a significant role in characterising the instability of orbits, the growth of trajectories, the sensitive dependence of the system on parametric disturbances in discrete and continuous dynamics, \textit{etc}. Ideas more than half-a-century old by Benettin \textit{et al} \cite{Benettin} regarding the relation between Lyapunov exponents and exponential divergence of typical trajectories have been used by a broad spectrum of mathematicians and physicists to understand  various dynamical systems. 
\medskip 

\noindent 
Pesin furthered this study in \cite{Pesin} to calculate the link between the Lyapunov exponents and the degree of stochasticity for typical trajectories and its relation to other measures of randomness like the Kolmogorov entropy.  A vast literature is available in this area, authored by several luminaries in the field on various types of dynamical systems, like polynomial or rational maps restricted on their Julia sets, Axiom A diffeomorphisms, Anosov flows,  \textit{etc}. A particularly desirable feature of all such maps is that they could be studied through a symbolic model, a property  that we thoroughly exploit in the computations in this paper.
\medskip 

\noindent 
Here, we focus on monic, centered, hyperbolic, quadratic and cubic polynomial maps $P$ with complex coefficients of modulus strictly less than one except the leading coefficient, restricted on its Julia set, $\mathcal{J}_{P}$, \textit{i.e.}, $P_{c} (z) = z^{2} + c$ and $P_{\mathbf{a}} (z) = z^{3} + a_{1} z + a_{0}$ with the coefficients $c$ and $\mathbf{a} = (a_{1}, a_{0})$ satisfying the other technical necessities, as mentioned above. Observe that our prescription insists that the parameter $c$ in the quadratic polynomial comes from $\mathbb{M}$, the \textit{Mandelbrot set} (see \cite{Beardon, Garijo, Lyubich} for more details) with $|c| < 1$. 
\medskip 

\noindent 
Manning proved in \cite{Manning} that the Lyapunov exponent of the quadratic polynomial map restricted on its Julia set with respect to the Lyubich's measure is constant. In fact, he deduces in that paper that the equilibrium distribution on the Julia set has Hausdorff dimension one. Sridharan proved in \cite{Sridharan} that there exists an infinitude of probability measures supported on the Julia set associated to various Bernoulli measures where Manning's results do not generalise, albeit for a very small class of quadratic polynomial maps, \textit{i.e.}, when $c \in [0, \frac{1}{4})$. Although not mentioned explicitly, the work of Sridharan in \cite{Sridharan} asserts a non-constant and a non-linear behaviour of the appropriate Lyapunov exponent (considered upto a suitable order) when the hyperbolic quadratic family of polynomials undergoes small but only real and positive perturbations in the above mentioned interval. 
\medskip 

\noindent 
Our aim in this paper is to calculate the Lyapunov exponent for a wide range of quadratic and cubic polynomial maps. Since the case considered under study in \cite{Sridharan} already establishes a non-linear behaviour of the Lyapunov exponent for the quadratic family of polynomials under small real positive perturbations, it makes no case for us to extend these results. However, we compute the Lyapunov exponent for a wider range of quadratic polynomials and cubic polynomials and obtain the values of the Lyapunov exponents computed in \cite{Sridharan}, as a corollary when $c \in [0, \frac{1}{4})$. A careful reader may observe that even though the computations of the necessary integrals may look similar to the ones in \cite{Sridharan}, we accommodate $c$ from a larger set and hence, the method of computations are sufficiently different. Further, once we compute the value of the Lyapunov exponent in the quadratic case, it gives us motivation to compute the quantity for cubic polynomials. 
\medskip 

\noindent 
This paper is organized thus. In the next section namely section \eqref{setting}, we write all necessary definitions that would lead us up to make a mathematically meaningful statement of the main results of this paper for the considered families of quadratic and cubic polynomials, as explained in the previous paragraph. In section \eqref{pressure}, we state necessary results from the literature that we use in the sequel. In the section \eqref{computation}, we compute the value of the Lyapunov exponents with respect to the appropriate measures of quadratic polynomials and cubic polynomials and prove the main results of this paper. We conclude the paper with a few observations regarding the so computed derivatives of the Lyapunov exponents, in section \eqref{concl}. 

\section{Basic settings and the Main results} 
\label{setting}

\noindent 
Let $\widehat{\mathbb{C}}$ denote the complex sphere, meaning the complex plane along with the  point at $\infty$. Let $P:\widehat{\mathbb{C}} \longrightarrow \widehat{\mathbb{C}}$ be a polynomial map of degree $d$: 
\[ P(z)\ \ :=\ \ b_{d} z^{d} + b_{d - 1} z^{d - 1} + \cdots + b_{1} z + b_{0}, \] 
where $b_{i} \in \mathbb{C},\ \forall i = 0, 1, \cdots, d$, with $b_{d} \neq 0$. By a set of affine transformations (and a slight abuse of notations), one can write the polynomial map $P$ in one of its normal forms (also referred to as $P$); monic and centered:
\[ P_{(a_{d - 2}, a_{d - 3}, \cdots, a_{1}, a_{0})} (z)\ \ =\ \ z^{d} + a_{d - 2} z^{d - 2} + \cdots + a_{1} z + a_{0}, \]
where $a_{i} \in \mathbb{C},\ \forall i = 0, 1, \cdots, d - 2$. Thus, we shall consider quadratic and cubic polynomials to look like, 
\begin{equation} 
\label{normalform}
P_{c} (z)\ \ =\ \ z^{2} + c\ \ \ \ \text{and}\ \ \ \  P_{(a_{1}, a_{0})} (z)\ \ =\ \ P_{\mathbf{a}} (z)\ \ =\ \ z^{3} + a_{1} z + a_{0}. 
\end{equation} 
One of the several possible definitions of \emph{Julia set} of a polynomial map $P$ states that it is the closure of the set of repelling periodic points of $P$, \textit{i.e.},
\[\mathcal{J}_{P}\ \ :=\ \ \overline{\left\{ z_{0} \in \widehat{\mathbb{C}} : P^{n} (z_{0}) = z_{0}\ \text{and}\ |(P^{n})' (z_{0})| > 1 \right\}}. \] 
It is then easy to observe that the Julia set $\mathcal{J}_{P}$ is a non-empty, compact, completely $P$-invariant metric space, where the family of iterates $\left\{P^{n}\right\}_{n \geq 0}$ is not normal (in the sense of Montel). For more properties of the Julia set of a polynomial map, interested readers are referred to \cite{Beardon, Lyubich}. 
\medskip 

\noindent 
In this paper, we will be interested in \emph{hyperbolic} Julia sets, \textit{i.e.}, $\exists$ constants $C > 0$ and $\lambda > 1$ such that $\inf_{z \in \mathcal{J}_{P}} \left\{ | P^{n} (z) | \right\} \ge C \lambda^{n},\ \forall n \geq 1$. Observe that a hyperbolic Julia set $\mathcal{J}_{P}$ is topologically connected. Further, this condition is realised on the quadratic family (when written in its normal form, as in \eqref{normalform}) when $c \in \mathbb{M}$, the Mandelbrot set that does not allow the critical orbit to approach $\infty$. In the case of the cubic polynomial (in its normal form as in \eqref{normalform}), the critical point is determined by the parameter $a_{1}$, while the critical orbit is determined by the parameters $a_{1}$ and $a_{0}$. We require that this critical orbit remains bounded.  We should  further need a technical condition that all the parameters of the quadratic and cubic polynomial $c$ and $\mathbf{a}$ respectively satisfy $|c| < 1$ and $|a_{i}| < 1$, for $i = 0, 1$. This technical condition becomes essential in our computations, later, as one may observe in section \eqref{computation}. 
\medskip 

\noindent 
We denote by $\mathcal{Q}$ and $\mathcal{Q}'$, the collection of all quadratic and cubic polynomials respectively with hyperbolic Julia sets, that we are interested in, as explained in the above paragraph, \textit{i.e.},
\begin{eqnarray} 
\mathcal{Q} & := & \left\{ P_{c} (z)\ :=\ z^{2} + c\ :\ \mathcal{J}_{P_{c}} \equiv \mathcal{J}_{c}\ \text{is hyperbolic and}\ |c| < 1 \right\}; \\ 
\mathcal{Q}' & := & \bigg\{ P_{\mathbf{a}} (z)\ :=\ z^{3} + a_{1}z + a_{0}\ :\ \mathcal{J}_{P_{\mathbf{a}}} \equiv \mathcal{J}_{\mathbf{a}}\ \text{is hyperbolic} \nonumber \\ & & \ \ \ \ \ \ \ \ \ \ \ \ \ \ \ \ \ \ \ \ \ \ \ \ \ \ \ \ \ \ \ \ \ \ \ \ \ \ \ \ \ \ \ \ \ \ \ \ \ \ \ \ \ \text{and}\ |a_{i}| < 1\ \text{for}\ i = 0, 1 \bigg\}. 
\end{eqnarray} 
For any arbitrary polynomial $P$, owing to the density of preimages of any generic point  $\zeta \in \mathcal{J}_{P}$, it can be observed that the sequence of measures 
\[ \mu_{n}^{\zeta}\ \ :=\ \ \frac{1}{d^{n}} \sum_{P^{n} (w) = \zeta} \delta_{w},\ \text{where}\ \delta_{w}\ \text{is the Dirac delta measure at the point}\ w,\]
converges to some measure $\mu$ called the \emph{Lyubich's measure}, independent of $\zeta$, in the weak*-topology. For example, suppose all the coefficients (except the leading one) are identically zero, then the quadratic polynomial looks like $P_{0}(z) = z^{2}$ while the cubic polynomial looks like $P_{\mathbf{0}} (z) = z^{3}$; both of which has the unit circle, $\mathit{S}^{1}$ in $\mathbb{C}$ as its Julia set, and the Lyubich's measure can be thought of as the \emph{Haar measure} on $\mathit{S}^{1}$. The support of the Lyubich's measure is the Julia set. 
\medskip 

\noindent 
The \emph{Lyapunov exponent} of a polynomial map $P$ restricted on its Julia set $\mathcal{J}_P$ with respect to any probability  measure $\mu$ supported on $\mathcal{J}_P$ is defined as:
\begin{equation} 
\label{L}
\mathcal{L}_{\mu} (P)\ \ :=\ \ - \int_{\mathcal{J}_P} \log |P'| d \mu.
\end{equation}
In order to calculate the Lyapunov exponent of the family of quadratic polynomial maps $P_{c} \in \mathcal{Q}$ restricted on its Julia set $\mathcal{J}_{c}$, we make use of the shift space and the shift map on $2$ symbols, whereas to calculate the Lyapunov exponent of the family of cubic polynomial maps $P_{\mathbf{a}} \in \mathcal{Q}'$ restricted on its Julia set $\mathcal{J}_{\mathbf{a}}$, we make use of the shift space and the shift map on $3$ symbols. We now briefly narrate the shift space and the shift map on $ d $ symbols. 
\medskip 

\noindent 
Let $\Sigma_{d}^{+}\ :=\ \left\{ 1, 2, 3, \cdots, d \right\}^{\mathbb{Z}^{+}}$, \textit{i.e.}, the set of all infinite sequences of  $\{ 1, 2, 3, \cdots, d \}$ indexed by the positive integers while $\sigma : \Sigma_{d}^{+} \longrightarrow \Sigma_{d}^{+}$ is the shift map defined by $(\sigma \underline{x})_{n}\ =\ x_{n + 1}$. Then we know by a theorem due to Lyubich in \cite{Lyubich} that whenever the polynomial $P$ of degree $d$ is hyperbolic, there exists  conjugacies $\Psi$ from $\Sigma_{d}^{+}$ to $\mathit{S}^1$ and $\Phi_{P}$ from $\mathit{S}^1$ to $\mathcal{J}_{P}$ such that
\[ \begin{CD}
\Sigma_{d}^{+} @> \Psi >> \mathit{S}^1 @> \Phi_{P} >> \mathcal{J}_{P} \\
@V \sigma VV @VRVV @VPVV \\
\Sigma_{d}^{+} @> \Psi >> \mathit{S}^1 @> \Phi_{P} >> \mathcal{J}_{P}
\end{CD} \]
where $R(z) = z^{d}$ that has the unit circle as its Julia set. In other words, this means that by virtue of the conjugacies $\Psi$ and $\Phi_{P}$, we have that 
\begin{equation} 
\label{Conjugacy}
P \circ \Phi_{P} \circ \Psi\ \ =\ \ \Phi_{P} \circ R \circ \Psi\ \ =\ \ \Phi_{P} \circ \Psi \circ \sigma.
\end{equation}
We will write more about the second part of the commuting diagram, later in section \eqref{computation}, as theorems \eqref{Carleson,gamelin, Zinsmeister} and \eqref{Carleson,gamelin3}. It is well-known that on the shift space, the clopen cylinder sets form a semi-algebra that generates the Borel $\sigma$-algebra. Making use of any positive probability vector $p = \left( p_{1}, p_{2}, \cdots, p_{d} \right)$ with $p_{i} > 0,\ \forall i = 1, 2, \cdots, d$ and $\sum_{i = 1}^{d} p_{i} = 1$, we define the \emph{Bernoulli measure} $\tilde{\mu}_{\left( p_{1}, p_{2}, \cdots, p_{d} \right)}$ on the cylinder sets of $\Sigma_{d}^{+}$ as
\begin{equation}
\tilde{\mu}_{\left( p_{1}, p_{2}, \cdots, p_{d} \right)} \left( \left[ x_{k}, x_{k + 1}, \cdots, x_{l} \right] \right)\ \ :=\ \ \bigg( p_{1}^{\# \{ x_{i} = 1 \}} \bigg) \times \bigg( p_{2}^{\# \{ x_{i} = 2 \}} \bigg) \times \cdots \times \bigg( p_{d}^{\# \{x_{i} = d \}} \bigg), 
\end{equation}
for $k \le i \le l$. Observe that the equidistributed Bernoulli measure on two symbols has $p_{i} = \frac{1}{2}$, for $i = 1, 2$ and the equidistributed Bernoulli measure on three symbols has $p_{i} = \frac{1}{3}$, for $i = 1, 2, 3$. The equidistributed Bernoulli measure on $\Sigma_{2}^{+}$ and $\Sigma_{3}^{+}$ is nothing but the Lyubich's  measure $\mu$ supported on the respective Julia set $\mathcal{J}_{c}$ and $\mathcal{J}_{\mathbf{a}}$; by virtue of the conjugacy in \eqref{Conjugacy}. Further, we denote by $\mu_{(p_{1}, p_{2})}$ and $\mu_{(p_{1}, p_{2}, p_{3})}$ the probability measures defined on $\mathcal{J}_{c}$ and $\mathcal{J}_{\mathbf{a}}$ respectively, associated to the Bernoulli measure $\tilde{\mu}_{(p_{1}, p_{2})}$ and $\tilde{\mu}_{(p_{1}, p_{2}, p_{3})}$ via the conjugacy mentioned above.
\medskip 

\noindent 
We urge the reader to observe that when all $ p_{i} $s are not equidistributed in the Bernoulli measure, it results in some of the pre-image branches in the definition of Lyubich's measure gaining prominence over the other branches. However, since the set of pre-image branches is uniformly distributed in the Julia set, $\mathcal{J}$, for any generic point $z \in \mathcal{J}$, this only means that as $n$ increases the sections of the Julia set, $\mathcal{J}$, that gain prominence get tinier and tinier. This family of non-equidistributed Bernoulli measures is interesting to work with, especially when one of the $p_{i}$'s is extremely close to $1$, leaving the remainder of the  $p_{j}$'s to be arbitrarily close to $0$. In such a case the section of the Julia set corresponding to the preimage branch that gains prominence in the Julia set (with respect to this Bernoulli measure), eventually reduces to a point measure. In the following theorems, we compute the Lyapunov exponent of the polynomials with respect to the equidistributed Bernoulli measure and this interesting case of the Bernoulli measure that could eventually reduce to being a point measure. We now state the main results of this paper. 
\medskip 

\noindent 
\begin{theorem}  
\label{theorem1}
Let $P_{c} (z) = z^{2} + c \in \mathcal{Q}$. Then, the Lyapunov exponent of $P_{c}$ restricted on its Julia set $\mathcal{J}_{c}$ with respect to the measure $\mu_{(p_{1}, p_{2})}$ associated to the Bernoulli measure $\tilde{\mu}_{(p_{1}, p_{2})}$ is given by 
\begin{eqnarray*}
\mathcal{L}_{\mu_{(p_{1}, p_{2})}} (P_{c})  \begin{cases}
= - \log 2 \quad & \text{if} \quad p_{1} = p_{2} = \frac{1}{2}, \\ 
\to - \log 2 + c_{ \mathbb{R} } + \frac{3}{2} c_{ \mathbb{R} }^{2} - \frac{3}{2} c_{ \mathbb{I} }^{2}  \quad & \text{as} \quad p_{1} \uparrow 1\ \text{or}\ p_{2} \uparrow 1;
\end{cases}
\end{eqnarray*}
where $ c = c_{\mathbb{R}} + i c_{\mathbb{I}} $ with $ c_{\mathbb{R}}, c_{\mathbb{I}} \in \mathbb{R} $. 
\end{theorem} 
\medskip 

\noindent 
\begin{theorem} 
\label{theorem2}
Let $ P_{\mathbf{a}} (z) = z^{3} + a_{1} z + a_{0} \in \mathcal{Q}'$. Then, the Lyapunov exponent of $ P_{\mathbf{a}} $ restricted on its Julia set $ \mathcal{J}_{\mathbf{a}} $ with respect to the measure $ \mu_{(p_{1}, p_{2}, p_{3})} $ associated to the Bernoulli measure $ \tilde{\mu}_{(p_{1}, p_{2}, p_{3})} $ is given by 
\begin{eqnarray*}
\mathcal{L}_{\mu_{(p_{1}, p_{2}, p_{3})}} (P_\mathbf{a})  \begin{cases}
= & - \log 3 \quad \text{if} \quad p_{1} = p_{2} = p_{3}, \\ 
\to & - \log 3 + \frac{1}{2} a_{ 1, \mathbb{R} } + \frac{1}{2} a_{ 0, \mathbb{R} } + \frac{3}{4} a_{ 1, \mathbb{R} } a_{ 0, \mathbb{R} } - \frac{3}{4} a_{ 1, \mathbb{I} } a_{ 0, \mathbb{I} } + \frac{1}{4} a_{ 1, \mathbb{R} }^{2} \\
& - \frac{1}{4} a_{ 1, \mathbb{I} }^{2} + \frac{1}{2} a_{ 0, \mathbb{R} }^{2} - \frac{1}{2} a_{ 0, \mathbb{I} }^{2} + \frac{15}{16} a_{ 1, \mathbb{R} }^{2}a_{ 0, \mathbb{R} } - \frac{15}{16} a_{ 1, \mathbb{I} }^{2}a_{ 0, \mathbb{R} } + \frac{3}{2}a_{ 1, \mathbb{R}}a_{ 0, \mathbb{R} }^{2} \\ 
& -  \frac{3}{2}a_{ 1, \mathbb{R} }a_{ 0, \mathbb{I} }^{2} -  \frac{15}{8}a_{ 1, \mathbb{R} }a_{ 1, \mathbb{I} }a_{ 0, \mathbb{I} } -3a_{ 1, \mathbb{I} }a_{ 0, \mathbb{R} }a_{ 0, \mathbb{I} } - 12a_{ 1, \mathbb{R} }a_{ 1, \mathbb{I} }a_{ 0, \mathbb{R} }a_{ 0, \mathbb{I} } \\ & + 3a_{ 1, \mathbb{R} }^{2} a_{ 0, \mathbb{R} }^{2} -3a_{ 1, \mathbb{R} }^{2} a_{ 0, \mathbb{I} }^{2} - 3a_{ 1, \mathbb{I} }^{2} a_{ 0, \mathbb{R} }^{2} + 3a_{ 1, \mathbb{I} }^{2} a_{ 0, \mathbb{I} }^{2} \quad \text{as} \quad p_{i} \uparrow 1;
\end{cases}
\end{eqnarray*}
where $ a_{j} = a_{j, \mathbb{R}} + i a_{j, \mathbb{I}} $ for $ j = 1, 0 $ with 
$ a_{j, \mathbb{R}}, a_{j, \mathbb{I}} \in \mathbb{R} $.
\end{theorem} 

\section{Pressure Function}
\label{pressure}

\noindent 
In this section, we define the pressure function and recall some basic results that will be useful in our analysis. The results stated in this section are true for any general polynomial $P$ restricted on its Julia set $\mathcal{J}_{P}$. For any continuous function $f : \mathcal{J}_{P} \longrightarrow \mathbb{R}$, we define its \emph{pressure} by 
\begin{equation} 
\label{Pr}
\mathfrak{P} (f)\ \ :=\ \ \sup_{\mu \in \mathcal{M} (\mathcal{J}_{P})} \left\{ h_{\mu} (P) + \int_{\mathcal{J}_{P}} f d \mu \right\},
\end{equation} 
where the supremum is taken over the space of all $P$-invariant probability measures supported on $\mathcal{J}_{P}$, denoted by $\mathcal{M}(\mathcal{J}_{P})$. Further, $h_\mu(P)$ denotes the entropy of $P$ with respect to the measure $\mu$. For more properties of pressure, interested readers are referred to \cite{Walter}. The pressure function is also characterised as follows.
\medskip 

\noindent 
\begin{theorem}[\cite{Walter}] 
\label{er}
For any $f \in \mathcal{C} (\mathcal{J}_{P}, \mathbb{R})$,
\begin{equation}
\mathfrak{P}(f)\ \ =\ \ \lim_{n\to \infty} \frac{1}{n} \log \left( \sum_{z \in \mathcal{J}_{P} \cap {\rm Fix}_{n} (P)}e^{f(z) + f(Pz) + \cdots + f(P^{n - 1}z)} \right).
\end{equation}
\end{theorem}
\medskip 

\noindent 
An \emph{equilibrium state} for the function $f$ denoted by $\mu^{(f)} \in \mathcal{M} (\mathcal{J}_{P})$ is a measure realising its supremum in the definition of pressure, \eqref{Pr}. If the real-valued continuous function $f  \in \mathcal{C}^{\alpha} (\mathcal{J}_{P}, \mathbb{R})$ for some H\"{o}lder exponent $\alpha$, then the existence of its unique equilibrium state is guaranteed by Denker and Urbanski in \cite{Denker}. Observe that the measure of maximal entropy of $P$; $\nu$ is then nothing but $\mu^{(f\equiv 0)}$. We now state an important theorem due to Ruelle \cite{Ruelle} and as observed by Coelho \textit{et al} \cite{Coelho}.
\medskip 

\noindent 
\begin{theorem}(\cite{Ruelle, Coelho}) 
\label{4}
For any two H\"{o}lder continuous functions $f, g \in \mathcal{C}^{\alpha} (\mathcal{J}_{P}, \mathbb{R})$ and $|t|$ sufficiently small,
\begin{equation}
\frac{d}{dt} \mathfrak{P} ( g + tf )\bigg|_{t = 0}\ \ =\ \ \int f d \mu^{(g)}.
\end{equation}
\end{theorem} 
\medskip 

\noindent 
Two real-valued continuous functions $f$ and $g$ defined on $\mathcal{J}_{P}$ are said to be \emph{cohomologous} to each other with respect to the polynomial $P$ if there exists a real-valued continuous function $h$ defined on $\mathcal{J}_{P}$ such that $f + h = g + h \circ P$. On the space of H\"{o}lder continuous functions, the map $f \longmapsto \mathfrak{P}(f)$ is real analytic. By the term real analytic, we mean that given an analytic function $f_{s} (z) = \sum_{n \geq 0} \Psi_{n} (z) s^{n}$ where $|s| \leq \epsilon$, for some small  $\epsilon > 0$, one can express $\mathfrak{P}(f_{s})$ as a summation of terms involving powers of $s$. If $f$ and $g$ are H\"{o}lder continuous and $f$ is not cohomologous to a constant, then the function $t \longmapsto \mathfrak{P} (g + tf)$ is strictly convex and real analytic, \textit{i.e.}, $\mathfrak{P} (g + tf)$ can be expressed as a summation of terms involving powers of $t$. A consequence of the above property and theorem \eqref{4} is captured in the next result. 
\medskip 

\noindent 
\begin{corollary} 
\label{Holder dimension}
For any fixed H\"{o}lder continuous function $f$, the map $g \longmapsto \int f d \mu^{(g)}$ is real analytic.
\end{corollary} 
\medskip 

\noindent 
The study of dimension theory was necessitated by the fact that measure theory failed to distinguish between countable and nowhere dense uncountable sets. Hausdorff dimension successfully captured the idea of a fractional dimension that overcame this failing. The \emph{Hausdorff dimension} of any $X \subset \mathbb{C}$ is defined as,
\begin{equation}
\begin{aligned}
{\rm dim}_{H}(X)\ \ :=\ \ \inf \Bigg\{ s > 0\ :\ \lim_{\epsilon \to 0} \bigg( & \inf_{\mathcal{U}} \Big\{ \sum_{i} \left({\rm diam} U_{i}\right)^{s} : \\ 
& \mathcal{U} = \{U_{i}\}\ \text{ is an}\ \epsilon-\text{open cover of}\ X \Big\} \bigg) = 0\Bigg\}.
\end{aligned}
\end{equation} 
It is then a simple observation from the definition that any set $E \subset \mathbb{C}$ satisfying ${\rm dim}_{H} (E) < 2$ should have no area.
\medskip 

\noindent 
For purposes of this paper, we shall be interested in the family of real-valued H\"{o}lder continuous functions $f_{s} := - s \log |P'| \in \mathcal{C}^{\alpha} (\mathcal{J}_{P}, \mathbb{R})$ for $s \in [0, 2]$. Then by theorem \eqref{er}, we have 
\[ \mathfrak{P} (- s \log |P'|)\ \ =\ \ \lim_{n \to \infty} \frac{1}{n} \log \left( \sum_{z \in \mathcal{J}_{P} \cap {\rm Fix}_{n} (P)} \frac{1}{|(P^{n})'(z)|^{s}} \right).\] 
The following result due to Bowen and Ruelle provides the relationship between the pressure function and the Hausdorff dimension of $\mathcal{J}_{P}$. 
\medskip 

\noindent 
\begin{theorem}[\cite{Ruelle}]
The unique solution to the equation $\mathfrak{P}(f_{s}) = 0$ is the Hausdorff dimension of $\mathcal{J}_{P}$.
\end{theorem}

\section{Computation of Lyapunov Exponents} 
\label{computation}

\noindent 
The rest of the paper is devoted to explicit computation of the Lyapunov exponent of the quadratic and cubic polynomial belonging to $\mathcal{Q}$ and $\mathcal{Q}'$, restricted on its Julia set, with respect to various measures associated to the Bernoulli measures and necessary analysis to prove our theorems. 

\subsection{Quadratic Polynomials}

\noindent 
Here, we consider the family of polynomial maps $P_{c}$ parametrised  by the complex constant $c$ as before:
\[P_{c} (z) = z^{2} + c \in \mathcal{Q},\] 
restricted on its Julia set denoted by $\mathcal{J}_{P_{c}} \equiv \mathcal{J}_{c}$. We shall adopt the method of computation from \cite{Zinsmeister, Sridharan}. Observe that  equation \eqref{Conjugacy} now looks like: 
\begin{equation}
\label{Conjugacy-quadratic}
P_{c} \circ \Phi_{c} \circ \Psi\ \ =\ \ \Phi_{c} \circ P_{0} \circ \Psi\ \ =\ \ \Phi_{c} \circ \Psi \circ \sigma.
\end{equation}
Here $P_{0}(z) \equiv R(z) = z^{2}$ and $\Phi_{c} = \Phi_{P_{c}}.$ We now write a theorem and include a short proof, for the convenience of the reader. This result can be found in \cite{Zinsmeister} (Theorem 6.4), \cite{Lyubich} (Proposition 1.12) and \cite{CarlesonGamelin} (Theorem 4.1, its proof and the note that follows the proof). 
\medskip 

\noindent 
\begin{theorem} \cite{CarlesonGamelin, Lyubich, Zinsmeister}
\label{Carleson,gamelin, Zinsmeister} 
The map $ c \longmapsto \Phi_{c} $ is an analytic map whenever $\mathcal{J}_{c}$ is hyperbolic.
\end{theorem}

\begin{proof} 
Let the conjugacy map $\Phi_{c}(z) = z + \cdots$ satisfy $\Phi_{c}( P_{c}(z)) = P_{0} (\Phi_{c}(z)) = \Phi_{c}^{2}(z)$. For large $|z|$, we have $P_{c}^{n}(z) = z^{2^{n}}(1+ \cdots)$. Defining $\left(\Phi_{c}\right)_{n}(z) := (P_{c}^{n}(z))^{2^{-n}} = z (1 + \cdots)^{2^{-n}}$, we observe that the sequence $\left(\Phi_{c}\right)_{n}$ satisfies 
\begin{eqnarray} 
\label{Functionalequationford=2} 
\left[\left(\Phi_{c}\right)_{n}\right]^{2} (z) = \left(\Phi_{c}\right)_{n - 1} \circ P_{c} (z),
 \end{eqnarray}
and that the sequence $\left(\Phi_{c}\right)_{n} \to \Phi_{c} $ uniformly. Therefore, the uniform limit $ \Phi_{c} $ is a solution of the functional equation \eqref{Functionalequationford=2}, and thus satisfies equation \eqref{Conjugacy-quadratic}. Since $c \longmapsto P_{c}$ is analytic, the map $c \longmapsto \Phi_{c}$ is analytic too, whenever $\mathcal{J}_{c}$ is hyperbolic. 
\end{proof} 
\medskip 
 
\noindent 
Thus in fact, when $P_{c} \in \mathcal{Q}$, we will consider 
\[ \Phi_{c} (z)\ \ =\ \ z + \sum_{n \geq 1} \varphi_{n} (z) c^{n}.\]
Taking the appropriate sides in equation \eqref{Conjugacy-quadratic} gives
\[ \left( z^{2} + \sum_{n \geq 1} \varphi_{n} (z^{2}) c^{n} \right) - \left( z + \sum_{n \geq 1} \varphi_{n} (z) c^{n} \right)^{2} - c\ =\ 0.\]  
Then by merely comparing the coefficients, we obtain from the computations in \cite{Zinsmeister, Sridharan} that,
\begin{eqnarray} 
\label{Comparing}
\varphi_{1} (z) & = & - z \sum_{i_{1} \geq 1} \frac{1}{2^{i_{1}} z^{2^{i_1}}} \\ 
\label{Comparing2}
\varphi_{2} (z) & = & - z \sum_{i_3 \geq 1}\frac{1}{2 ^ {i_{3} }} \sum_{i_2 \geq 1}\sum_{i_1 \geq 1}^{i_2} \frac{1}{2^{i_{2}+1}}  \frac{1}{z^{2^{i_3}-2^{i_{3}-1}+2^{i_3-1}(2^{i_1}+2^{i_2-i_1+1}-1)}}  . 
\end{eqnarray}
We now consider the pressure of the function $f_{1} = - \log |P_{c}'|$ and find using the conjugacy in equation  \eqref{Conjugacy-quadratic} that
\begin{eqnarray*}
\mathfrak{P} (f_{1}) & = & \sup_{\mu \in \mathcal{M} (\mathcal{J}_{c})} \left\{ h_{\mu} (P_{c}) + \mathcal{L}_{\mu_{(p_{1}, p_{2})}} (P_{c}) \right\} 
 =  \sup_{\nu \in \mathcal{M} (\mathit{S}^{1})} \left\{ h_{\nu} (P_{0}) - \log 2 - \int_{\mathit{S}^{1}} \log |\Phi_{c} (z) | d \nu \right\}.
\end{eqnarray*}
Consider the $P_{c}$-invariant probability measure $\mu_{(p_{1}, p_{2})}$ associated to the Bernoulli measure $\tilde{\mu}_{(p_{1}, p_{2})}$ with $p_{1} + p_{2} = 1$ and $p_{i} > 0$ for $i = 1, 2$. A simple calculation in \cite{Walter} then says that the entropy of the polynomial $P_{c}$, namely $h_{\mu_{(p_{1}, p_{2})}} (P_{c})$ is $- \log (p_{1}^{p_{1}} p_{2}^{p_{2}}) = \log 2$ when $p_{1} = p_{2}$. 

\subsubsection*{Case - 1 : $c \in \mathbb{M} \cap \mathbb{R}$ with $|c| < 1$} 
\label{secondsubsection}

\noindent 
Here, we compute $\int \log |\Phi_{c} (z)| d \mu_{(p_{1}, p_{2})}$ when $c \in \mathbb{M} \cap \mathbb{R}$, and $ c \equiv c_{\mathbb{R} }$, satisfying $|c| < 1$. 
\begin{eqnarray} 
\label{|c|<1}
\int \log |\Phi_{c} (z)| d \mu_{(p_{1}, p_{2})} & = & c \left[ \int {\rm Re}\left( \overline{z} \varphi_{1} (z) \right) d \mu_{(p_{1}, p_{2})} \right] \nonumber \\
& & +\ c^{2} \Bigg[ \int \bigg( {\rm Re} \left( \overline{z} \varphi_{2} (z) \right)-\frac{1}{2} \left( {\rm Re}\left( \overline{z} \varphi_{1} (z) \right) \right)^{2} \nonumber \\ 
& & \ \ \ \ \ \ \ \ \ \ \ \ \ \ \ \ \ \ \ +\frac{1}{2} \left( {\rm Im}\left( \overline{z} \varphi_{1} (z) \right) \right)^{2}   \bigg) d \mu_{(p_{1}, p_{2})} \Bigg]  +\ O(c^{3}). 
\end{eqnarray} 
We now focus on the coefficients upto $c^{2}$ and evaluate the integrals with respect to the measure $\mu_{(p_{1}, p_{2})}$ associated to the Bernoulli measure $\tilde{\mu}_{(p_{1}, p_{2})}$. Making use of the computations in \eqref{Comparing}, \eqref{Comparing2}, we obtain
\begin{eqnarray*}
\int {\rm Re} \left( \overline{z} \varphi_{1} \right) d \mu_{(p_{1}, p_{2})} 
& & \begin{cases} 
= 0 \quad & \text{if} \quad p_{1} = p_{2} = \frac{1}{2}, \\ 
\to -1 \quad & \text{as}\ \quad p_{i} \uparrow 1.
\end{cases} \\
\int {\rm Re} \left( \overline{z} \varphi_{2} \right) d \mu_{(p_{1}, p_{2})} 
& & \begin{cases}
= 0 \quad & \text{if} \quad p_{1} = p_{2} = \frac{1}{2}, \\ 
\to -1 \quad & \text{as}\ \quad p_{i} \uparrow 1; 
\end{cases} \\
\frac{1}{2} \int \left( {\rm Im} \left( \overline{z} \varphi_{1} \right) \right)^{2} d \mu_{(p_{1}, p_{2})}  
& = & \ \ \ \ \ \ 0,\ \quad \text{irrespective of the value of}\ p_{i}; \\ 
 - \frac{1}{2} \int \left( {\rm Re} \left( \overline{z} \varphi_{1} \right) \right)^{2} d \mu_{(p_{1}, p_{2})} 
& & \begin{cases}
= 0 \quad & \text{if} \quad p_{1} = p_{2} = \frac{1}{2}, \\ 
\to - \frac{1}{2} \quad & \text{as}\ \quad p_{i} \uparrow 1. 
\end{cases}
\end{eqnarray*}
Recall that when $p_{1} = p_{2} = \frac{1}{2}$, the measure $\mu_{(p_{1}, p_{2})}$ associated to the Bernoulli measure $\tilde{\mu}_{(p_{1}, p_{2})}$ corresponds to the Haar measure on $\mathit{S}^1$. Therefore, by considering up to coefficients of $c^{2}$ in equation \eqref{|c|<1} and making use of the evaluation above, we infer that
\begin{eqnarray} 
 \mathcal{L}_{\mu_{(p_{1}, p_{2})}} (P_{c})  & = & - \log 2 - \int \log |\Phi_{c} (z)| d \mu_{(p_{1}, p_{2})} \nonumber \\ 
& & \begin{cases}
= - \log 2 \quad & \text{if} \quad p_{1} = p_{2} = \frac{1}{2}, \\ 
\to - \log 2 + c + \frac{3}{2} c^{2}   \quad & \text{as}\ \quad p_{1} \uparrow 1\ \text{or}\ \quad p_{2} \uparrow 1.
\end{cases}
\end{eqnarray} 

\subsubsection*{Case - 2 : $c \in \mathbb{M}$ with $|c| < 1$}

\noindent 
We begin the computations here with the observation that this case subsumes  the earlier one. We have  $c = c_{\mathbb{R}} + i c_{\mathbb{I}}$ with $c_{\mathbb{R}}, c_{\mathbb{I}} \in \mathbb{R}$, then $|c_{\mathbb{R}}| < 1$ and $|c_{\mathbb{I}}| < 1$. A computation similar to the earlier case then yields,
\begin{eqnarray} 
\label{second|c|<1}
\int \log | \Phi_{c} (z) | d \mu_{(p_{1}, p_{2})} & = & \int {\rm Re} \left( c \overline{z} \varphi_{1} (z) \right) d \mu_{(p_{1}, p_{2})} \nonumber \\
& & + \int \bigg[ {\rm Re} \left( c^{2} \overline{z} \varphi_{2} (z) \right) + \frac{1}{2} \left( {\rm Im} \left( c \overline{z} \varphi_{1} (z) \right) \right)^{2} \nonumber \\ 
& & \ \ \ \ \ \ \ \ \ \ \ \ \ \ \ \ -\frac{1}{2} \left( {\rm Re} \left( c \overline{z} \varphi_{1} (z) \right) \right)^{2} \bigg] d \mu_{(p_{1}, p_{2})}  +\ O(c^{3}).
\end{eqnarray}
Now, we focus on the coefficients upto $c^{2}$ and evaluate the integrals with respect to measure $\mu_{(p_{1}, p_{2})}$ associated to Bernoulli measure $\tilde{\mu}_{(p_{1}, p_{2})}$. Making use of the computations in \eqref{Comparing} and \eqref{Comparing2} yet again, we obtain 
\begin{eqnarray*}
\int {\rm Re} \left( c \overline{z} \varphi_{1} \right) d \mu_{(p_{1}, p_{2})}
& & \quad\quad\begin{cases}
= 0 \quad & \text{if} \quad p_{1} = p_{2} = \frac{1}{2}, \\ 
\to - c_{\mathbb{R}} \quad & \text{as}\ \quad p_{i} \uparrow 1. 
\end{cases} \\
-\int {\rm Re} \left( c^{2} \overline{z} \varphi_{2} \right) d \mu_{(p_{1}, p_{2})} 
& & \quad\quad\begin{cases}
= 0 \quad & \text{if} \quad p_{1} = p_{2} = \frac{1}{2}, \\ 
\to (c_{\mathbb{R}}^{2} - c_{\mathbb{I}}^{2}) \quad & \text{as}\ \quad p_{i} \uparrow 1; 
\end{cases} \\
\frac{1}{2} \int \left( {\rm Re} \left( c \overline{z} \varphi_{1} \right) \right)^{2} d \mu_{(p_{1}, p_{2})} 
& & \quad\quad\begin{cases}
= 0 \quad & \text{if} \quad p_{1} = p_{2} = \frac{1}{2}, \\ 
\to \frac{1}{2}c_{\mathbb{R}}^{2} \quad & \text{as}\ \quad p_{i} \uparrow 1; 
\end{cases} 
\end{eqnarray*} 
\begin{eqnarray*}
 -\ \frac{1}{2} \int \left( {\rm Im} \left( c \overline{z} \varphi_{1} \right) \right)^{2} d \mu_{(p_{1}, p_{2})}  
& & \quad\begin{cases}
= 0 \quad & \text{if} \quad p_{1} = p_{2} = \frac{1}{2}, \\ 
\to - \frac{1}{2} c_{\mathbb{I}}^{2} \quad & \text{as}\ \quad p_{i} \uparrow 1.  
\end{cases}
\end{eqnarray*} 

\noindent 
By considering upto coefficients of $c^{2}$ in equation \eqref{second|c|<1} and making use of the evaluations above, we infer that
\begin{eqnarray} 
\label{complexmatrix2}
 \mathcal{L}_{\mu_{(p_{1}, p_{2})}} (P_{c})  & = & - \log 2 - \int \log |\Phi_{c} (z)| d \mu_{(p_{1}, p_{2})} \nonumber \\ 
& &\begin{cases}
= - \log 2 \quad & \text{if} \quad p_{1} = p_{2} = \frac{1}{2}, \\ 
\to - \log 2 + c_{ \mathbb{R} } + \frac{3}{2} c_{ \mathbb{R} }^{2} - \frac{3}{2} c_{ \mathbb{I} }^{2} .  \quad & \text{as} \quad p_{1} \uparrow 1\ \text{or}\ p_{2} \uparrow 1.
\end{cases}
\end{eqnarray} 

\subsection{Cubic Polynomials} 
\label{cubicpolynomials}

\noindent 
In this section, we consider the monic, centered, cubic polynomial map $P_{\mathbf{a}}$ parametrised by the complex constants $a_{1}$ and $a_{0}$, given by 
\[ P_{\mathbf{a}} (z)\ \ \equiv\ \ P_{(a_{1}, a_{0})} (z)\ \ =\ \ z^{3} + a_{1} z + a_{0}\ \ \in\ \ \mathcal{Q}',\] 
restricted on its Julia set, denoted by $\mathcal{J}_{P_{\mathbf{a}}} \equiv \mathcal{J}_{\mathbf{a}}$. This form of the cubic polynomial easily identifies the critical points of the polynomial $P_{\mathbf{a}}$. As mentioned in section \eqref{setting}, we will impose a technical condition on the complex coefficients, $a_{1}$ and $a_{0}$ of $P_{\mathbf{a}}$ that $|a_{1}| < 1,\ |a_{0}| < 1$. An imminently stricter condition may also be required to make the corresponding Julia set hyperbolic, see \cite{Alexander, Alexander Blokh}. 
\medskip 

\noindent 
Observe that equation \eqref{Conjugacy} now looks like 
\begin{equation} 
\label{Conjugacy-cubic}
P_{\mathbf{a}} \circ \Phi_{\mathbf{a}} \circ \Psi\ \ =\ \ \Phi_{\mathbf{a}} \circ P_{\mathbf{0}} \circ \Psi\ \ =\ \ \Phi_{\mathbf{a}} \circ \Psi \circ \sigma. 
\end{equation}
Here $P_{\mathbf{0}} (z) \equiv R(z) = z^{3}$ and $\Phi_{\mathbf{a}} = \Phi_{P_{\mathbf{a}}}$. 
\medskip 

\noindent 
\begin{theorem} \cite{CarlesonGamelin}
\label{Carleson,gamelin3}
The map $\mathbf{a} \longmapsto \Phi_{\mathbf{a}}$ is an analytic map whenever $\mathcal{J}_{\mathbf{a}}$ is hyperbolic.
\end{theorem}

\noindent 
\begin{proof} 
For a fixed $z$ and $a_{0}$, consider the map $a_{1} \longmapsto \Phi_{\mathbf{a}}$. Then, going through the same ideas as in the proof of theorem \eqref{Carleson,gamelin, Zinsmeister}, we infer that $a_{1} \longmapsto \Phi_{\mathbf{a}}$ is analytic, whenever $\mathcal{J}_{\mathbf{a}}$ is hyperbolic. Similarly, for a fixed $z$ and $a_{1}$, we obtain the map $a_{0} \longmapsto \Phi_{\mathbf{a}}$ is analytic too. 
\end{proof} 
\medskip 
 
\noindent
By our choice of the polynomial $P_{\mathbf{a}} \in \mathcal{Q}'$ and theorem \eqref{Carleson,gamelin3}, it is clear that the conjugacy map $\Phi_{\mathbf{a}}$ is analytic in both the parameters $a_{1}$ and $a_{0}$. Therefore, it is reasonable to consider 
\[ \Phi_{\mathbf{a}} (z)\ \ =\ \ z + \sum_{i + j \geq 1} \varphi_{(i, j)} (z) a_{1}^{i} a_{0}^{j}. \]
Then, the appropriate sides in equation \eqref{Conjugacy-cubic} gives 
\begin{eqnarray} 
\label{coefficient3}
\left( z^{3} + \sum_{i + j \geq 1} \varphi_{(i, j)} (z^{3}) a_{1}^{i} a_{0}^{j} \right)\ -\ \left( z + \sum_{i + j \geq 1} \varphi_{(i, j)} (z) a_{1}^{i}a_{0}^{j} \right)^{3} & & \nonumber \\ 
- a_{1} \left( z +\sum_{i + j \geq 1} \varphi_{(i, j)} (z) a_{1}^{i} a_{0}^{j} \right) - a_{0} & = & 0.
\end{eqnarray}
As earlier, a comparison of coefficients in \eqref{coefficient3} gives the necessary functions. Since we are only interested in terms upto quadratic order in the Lyapunov exponent and would accumulate terms of order three or more in $O(a_{1}^{3}, a_{0}^{3})$, we observe that it is sufficient for us to obtain a representation of the functions, $\varphi_{(1, 0)},\ \varphi_{(0, 1)},\ \varphi_{(1, 1)},\ \varphi_{(2, 0)},\ \varphi_{(0, 2)},\ \varphi_{(2, 1)},\ \varphi_{(1, 2)}$ and $\varphi_{(2, 2)}$. In fact, we obtain these functions explicitly as a series (as earlier), though we do not write them here, since the writing of the same could get extremely messy. Since we deal with the family of cubic polynomials here, it is only but natural that the Bernoulli measure is now considered with three parameters. Therefore, the Lyapunov exponent of the polynomial map $P_{\mathbf{a}}$ with respect to the $P_{\mathbf{a}}$-invariant measure $\mu_{(p_{1}, p_{2}, p_{3})}$  from equation \eqref{L} is 
\begin{equation} 
\label{cubicL} 
\mathcal{L}_{\mu_{(p_{1}, p_{2}, p_{3})}} \left(P_{\mathbf{a}}\right)\ \ =\ \ - \int_{\mathcal{J}_{\mathbf{a}}} \log |3 z^{2} + a_{1} | d \mu_{(p_{1}, p_{2}, p_{3})}. 
\end{equation} 
However, we know that in order to compute the dependence of the integral mentioned in the right hand side of equation \eqref{cubicL} on the coefficients $a_{1}$ and $a_{0}$, it is sufficient for us to compute the integral $-\int_{\mathit{S}^1} \log | \Phi_{\mathbf{a}} | d \mu_{(p_{1}, p_{2}, p_{3})}$. 
\begin{eqnarray} 
\label{cubiclyapcong}
& & - \int_{\mathit{S}^{1}} \log | \Phi_{\mathbf{a}} | d \mu_{(p_{1}, p_{2}, p_{3})} \\ 
\label{a1} 
& = & - \int_{\mathit{S}^{1}} \left[ {\rm Re} \left( a_{1} \overline{z} \varphi_{(1, 0)} \right) \right] d \mu_{(p_{1}, p_{2}, p_{3})} \\ 
\label{a0} 
& &  - \int_{\mathit{S}^{1}} \left[ {\rm Re} \left( a_{0} \overline{z} \varphi_{(0, 1)} \right) \right] d \mu_{(p_{1}, p_{2}, p_{3})} \\ 
& & - \int_{\mathit{S}^{1}} \Big[ \ {\rm Re} \left( a_{1} a_{0} \overline{z} \varphi_{(1, 1)} \right) - \left\{ {\rm Re}\left( \bar{z} \varphi_{(1,0)} a_{1} \right) \times {\rm Re} \left( \bar{z} \varphi_{(0,1)} a_{0} \right) \right\} \nonumber \\
\label{a1a0} 
& &\ \ \ \ \ \ \ +\left\{ {\rm Im}\left( \bar{z} \varphi_{(1,0)} a_{1} \right) \times {\rm Im} \left( \bar{z} \varphi_{(0,1)} a_{0} \right) \right\} \Big] d \mu_{(p_{1}, p_{2}, p_{3})} \\ 
& & \ \ \ \ \ \ \ \  \ \ \ \ \ \ \ \ \ \ \ \ \ \ \text{{\footnotesize (accounting for the terms with linear powers namely $a_{1},\ a_{0}$ and $a_{1} a_{0}$)}} \nonumber 
\end{eqnarray} 
\begin{eqnarray} 
& & - \int_{\mathit{S}^{1}} \left[ {\rm Re} \left( a_{1}^{2} \overline{z} \varphi_{(2, 0)} \right) - \frac{1}{2} \left\{ {\rm Re} \left( a_{1} \overline{z} \varphi_{(1, 0)} \right) \right\}^{2}+ \frac{1}{2} \left\{ {\rm Im} \left( a_{1} \overline{z} \varphi_{(1, 0)} \right) \right\}^{2} \right] d \mu_{(p_{1}, p_{2}, p_{3})} \nonumber \\ 
\label{a1sq} 
& & \ \ \ \ \ \ \ \  \ \ \ \ \ \ \ \ \ \ \ \ \ \ \ \ \ \ \ \ \ \ \ \ \ \ \ \ \ \ \ \ \ \ \ \ \ \ \ \ \ \ \ \ \ \ \ \ \ \ \ \ \ \ \ \ \ \ \ \ \ \ \ \text{{\footnotesize (terms involving $a_{1}^{2}$)}} \\
& & - \int_{\mathit{S}^{1}} \left[ {\rm Re} \left( a_{0}^{2} \overline{z} \varphi_{(0, 2)} \right) - \frac{1}{2} \left\{ {\rm Re} \left( a_{0} \overline{z} \varphi_{(0, 1)} \right) \right\}^{2} + \frac{1}{2} \left\{ {\rm Im} \left( a_{0} \overline{z} \varphi_{(0, 1)} \right) \right\}^{2} \right] d \mu_{(p_{1}, p_{2}, p_{3})} \nonumber \\ 
\label{a0sq} 
& & \ \ \ \ \ \ \ \  \ \ \ \ \ \ \ \ \ \ \ \ \ \ \ \ \ \ \ \ \ \ \ \ \ \ \ \ \ \ \ \ \ \ \ \ \ \ \ \ \ \ \ \ \ \ \ \ \ \ \ \ \ \ \ \ \ \ \ \ \ \ \ \text{{\footnotesize (terms involving $a_{0}^{2}$)}} 
\end{eqnarray} 
\begin{eqnarray} 
& & - \int_{\mathit{S}^{1}} \bigg[ {\rm Re} \left( a_{1}^{2} a_{0} \overline{z} \varphi_{(2, 1)} \right) \\ 
& &\ \ \ \ \ \ \ \ - \left\{ {\rm Re} \left( a_{1} \overline{z} \varphi_{(1, 0)} \right) \times {\rm Re} \left( a_{1} a_{0} \overline{z} \varphi_{(1, 1)} \right) \right\} + \left\{ {\rm Im} \left( a_{0} \overline{z} \varphi_{(0, 1)} \right) \times {\rm Im} \left( a_{1}^{2} \overline{z} \varphi_{(2, 0)} \right) \right\} \nonumber\\
& &\ \ \ \ \ \ \ \ +\left\{ {\rm Im} \left( a_{1} \overline{z} \varphi_{(1, 0)} \right) \times {\rm Im} \left( a_{1} a_{0} \overline{z} \varphi_{(1, 1)} \right) \right\} - \left\{ {\rm Re} \left( a_{0} \overline{z} \varphi_{(0, 1)} \right) \times {\rm Re} \left( a_{1}^{2} \overline{z} \varphi_{(2, 0)} \right) \right\} \nonumber \\ 
& & \ \ \ \ \ \ \ \ + \left\{ \left\{ {\rm Re} \left( a_{1} \overline{z} \varphi_{(1, 0)} \right) \right\}^{2} \times {\rm Re} \left( a_{0} \overline{z} \varphi_{(0, 1)} \right) \right\} - \left\{ {\rm Re} \left( a_{0} \overline{z} \varphi_{(0, 1)} \right) \times \left( {\rm Im} \left( a_{1}  \overline{z} \varphi_{(1, 0)} \right) \right)^{2}\right\}\nonumber\\
& & \ \ \ \ \ \ \ \ - 2 \left\{ {\rm Re} \left( a_{1} \overline{z} \varphi_{(1,0)} \right) \times {\rm Im} \left( a_{0} \overline{z} \varphi_{(0, 1)} \right) \times {\rm Im} \left( a_{1} \overline{z} \varphi_{(1, 0)} \right) \right\}\bigg] d \mu_{(p_{1}, p_{2}, p_{3})} \nonumber \\ 
\label{a1sqa0} 
& & \ \ \ \ \ \ \ \  \ \ \ \ \ \ \ \ \ \ \ \ \ \ \ \ \ \ \ \ \ \ \ \ \ \ \ \ \ \ \ \ \ \ \ \ \ \ \ \ \ \ \ \ \ \ \ \ \ \ \ \ \ \ \ \ \ \ \ \ \text{{\footnotesize (terms involving $a_{1}^{2} a_{0}$)}} 
\end{eqnarray} 
\begin{eqnarray} 
& & - \int_{\mathit{S}^{1}} \bigg[ {\rm Re} \left( a_{1}a_{0}^2 \overline{z} \varphi_{(1, 2)} \right) \\ 
& &\ \ \ \ \ \ \ \ - \left\{ {\rm Re} \left( a_{1} \overline{z} \varphi_{(1, 0)} \right) \times \left( {\rm Im} \left( a_{0}  \overline{z} \varphi_{(0, 1)} \right) \right)^{2} \right\} - \left\{ {\rm Re}\left( a_{1} \overline{z} \varphi_{(1, 0)} \right) \times {\rm Re} \left( a_{0}^{2} \overline{z} \varphi_{(0, 2)} \right) \right\}\nonumber \\
& & \ \ \ \ \ \ \ \ - \left\{ {\rm Re} \left( a_{0} \overline{z} \varphi_{(0, 1)} \right) \times {\rm Re} \left( a_{1} a_{0} \overline{z} \varphi_{(1, 1)} \right) \right\} +\left\{ {\rm Im} \left( a_{0} \overline{z} \varphi_{(0, 1)} \right) \times {\rm Im} \left( a_{1} a_{0} \overline{z} \varphi_{(1, 1)} \right) \right\}\nonumber\\
& & \ \ \ \ \ \ \ \ + \left\{ \left\{ {\rm Re} \left( a_{0} \overline{z} \varphi_{(0, 1)} \right)\right\}^{2} \times {\rm Re} \left( a_{1} \overline{z} \varphi_{(1, 0)} \right) \right\} + \left\{ {\rm Im}\left( a_{1} \overline{z} \varphi_{(1, 0)} \right) \times {\rm Im} \left( a_{0}^{2} \overline{z} \varphi_{(0, 2)} \right) \right\} \nonumber\\
 & & \ \ \ \ \ \ \ \ - 2 \left\{ {\rm Re} \left( a_{0} \overline{z} \varphi_{(0,1)} \right) \times {\rm Im} \left( a_{1} \overline{z} \varphi_{(1, 0)} \right) \times {\rm Im} \left( a_{0} \overline{z} \varphi_{(0, 1)} \right) \right\} \bigg] d \mu_{(p_{1}, p_{2}, p_{3})} \nonumber \\ 
\label{a1a0sq} 
& & \ \ \ \ \ \ \ \  \ \ \ \ \ \ \ \ \ \ \ \ \ \ \ \ \ \ \ \ \ \ \ \ \ \ \ \ \ \ \ \ \ \ \ \ \ \ \ \ \ \ \ \ \ \ \ \ \ \ \ \ \ \ \ \ \ \ \ \ \ \ \text{{\footnotesize (terms involving $a_{1} a_{0}^{2}$)}} 
\end{eqnarray} 
\begin{eqnarray} 
& & - \int_{\mathit{S}^{1}} \bigg[ {\rm Re} \left( a_{1}^{2} a_{0}^{2} \overline{z} \varphi_{(2, 2)} \right) -  \frac{1}{2} \left\{ {\rm Re} \left( a_{1} a_{0} \overline{z} \varphi_{(1, 1)} \right) \right\}^{2}+\frac{1}{2} \left\{ {\rm Im} \left( a_{1} a_{0} \overline{z} \varphi_{(1, 1)} \right) \right\}^{2} \nonumber \\ 
& & \ \ \ \ \ \ \ \ - \left\{ {\rm Re} \left( a_{1} \overline{z} \varphi_{(1, 0)} \right) \times {\rm Re} \left( a_{1} a_{0}^{2} \overline{z} \varphi_{(1, 2)} \right) \right\} +\left\{ {\rm Im} \left( a_{1} \overline{z} \varphi_{(1, 0)} \right) \times {\rm Im} \left( a_{1} a_{0}^{2} \overline{z} \varphi_{(1, 2)} \right) \right\}\nonumber\\
& & \ \ \ \ \ \ \ \ - \left\{ {\rm Re} \left( a_{0} \overline{z} \varphi_{(0, 1)} \right) \times {\rm Re} \left( a_{1}^{2} a_{0} \overline{z} \varphi_{(2, 1)} \right) \right\} +\left\{ {\rm Im} \left( a_{0} \overline{z} \varphi_{(0, 1)} \right) \times {\rm Im} \left( a_{1}^{2} a_{0} \overline{z} \varphi_{(2, 1)} \right) \right\}\nonumber\\
& & \ \ \ \ \ \ \ \ - \left\{ {\rm Re} \left( a_{1}^{2} \overline{z} \varphi_{(2, 0)} \right) \times {\rm Re} \left( a_{0}^{2} \overline{z} \varphi_{(0, 2)} \right) \right\} +\left\{ {\rm Im} \left( a_{1}^{2} \overline{z} \varphi_{(2, 0)} \right) \times {\rm Im} \left( a_{0}^{2} \overline{z} \varphi_{(0, 2)} \right) \right\}\nonumber\\
& & \ \ \ \ \ \ \ \ + \left\{ \left\{ {\rm Re} \left( a_{1} \overline{z} \varphi_{(1, 0)} \right) \right\}^{2} \times {\rm Re} \left( a_{0}^{2} \overline{z} \varphi_{(0, 2)} \right) \right\} + \left\{ \left\{ {\rm Re} \left( a_{0} \overline{z} \varphi_{(0, 1)} \right) \right\}^{2} \times {\rm Re} \left( a_{1}^{2} \overline{z} \varphi_{(2, 0)} \right) \right\} \nonumber \\ 
& & \ \ \ \ \ \ \ \ + 2 \left\{ {\rm Re} \left( a_{1} \overline{z} \varphi_{(1,0)} \right) \times {\rm Re} \left( a_{0} \overline{z} \varphi_{(0, 1)} \right) \times {\rm Re} \left( a_{1} a_{0} \overline{z} \varphi_{(1, 1)} \right) \right\} \nonumber \\
& & \ \ \ \ \ \ \ \ -2\left\{\left\{ {\rm Re} \left( a_{1} \overline{z} \varphi_{(1, 0)}\right)\right\} \times \left\{ {\rm Im} \left( a_{1} \overline{z} \varphi_{(1,0)}\right)\right\} \times  \left\{ {\rm Im} \left( a_{0}^{2} \overline{z} \varphi_{(0,2)}\right)\right\} \right\}\nonumber\\
& & \ \ \ \ \ \ \ \ -2\left\{\left\{ {\rm Re} \left( a_{1} \overline{z} \varphi_{(1,0)}\right)\right\} \times \left\{ {\rm Im} \left( a_{0} \overline{z} \varphi_{(0,1)}\right)\right\} \times  \left\{ {\rm Im} \left( a_{1}a_{0} \overline{z} \varphi_{(1,1)}\right)\right\} \right\}\nonumber\\
& & \ \ \ \ \ \ \ \ -2\left\{\left\{ {\rm Re} \left( a_{0} \overline{z} \varphi_{(0,1)}\right)\right\} \times \left\{ {\rm Im} \left( a_{1} \overline{z} \varphi_{(1,0)}\right)\right\} \times  \left\{ {\rm Im} \left( a_{1}a_{0} \overline{z} \varphi_{(1,1)}\right)\right\} \right\}\nonumber\\
& & \ \ \ \ \ \ \ \ -2\left\{\left\{ {\rm Re} \left( a_{0} \overline{z} \varphi_{(0,1)}\right)\right\} \times \left\{ {\rm Im} \left( a_{0} \overline{z} \varphi_{(0,1)}\right)\right\} \times  \left\{ {\rm Im} \left( a_{1}^{2} \overline{z} \varphi_{(2,0)}\right)\right\} \right\}\nonumber\\
& & \ \ \ \ \ \ \ \ -2\left\{\left\{ {\rm Re} \left( a_{1}a_{0} \overline{z} \varphi_{(1,1)}\right)\right\} \times \left\{ {\rm Im} \left( a_{1} \overline{z} \varphi_{(1,0)}\right)\right\} \times  \left\{ {\rm Im} \left( a_{0} \overline{z} \varphi_{(0,1)}\right)\right\} \right\}\nonumber\\
& & \ \ \ \ \ \ \ \ -\left\{\left\{ {\rm Re} \left( a_{1}^{2} \overline{z} \varphi_{(2,0)}\right)\right\} \times \left\{ {\rm Im} \left( a_{0} \overline{z} \varphi_{(0,1)}\right)\right\}^2 + \left\{ {\rm Re} \left( a_{0}^{2} \overline{z} \varphi_{(0,2)}\right)\right\} \times \left\{ {\rm Im} \left( a_{1} \overline{z} \varphi_{(1,0)}\right)\right\}^2\right\}\nonumber\\
& & \ \ \ \ \ \ \ \ -\frac{3}{2}\left\{ \left\{ {\rm Re} \left( a_{1} \overline{z} \varphi_{(1, 0)} \right) \right\}^{2} \times \left\{ {\rm Re} \left( a_{0} \overline{z} \varphi_{(0, 1)} \right) \right\}^{2} + \left\{ {\rm Im} \left( a_{1} \overline{z} \varphi_{(1, 0)} \right) \right\}^{2} \times \left\{ {\rm Im} \left( a_{0} \overline{z} \varphi_{(0, 1)} \right) \right\}^{2}\right\}\nonumber\\
& &\ \ \ \ \ \ \ \ +\frac{3}{2}\left\{ \left\{ {\rm Re} \left( a_{1} \overline{z} \varphi_{(1, 0)} \right) \right\}^{2} \times \left\{ {\rm Im} \left( a_{0} \overline{z} \varphi_{(0, 1)} \right) \right\}^{2} + \left\{ {\rm Re} \left( a_{0} \overline{z} \varphi_{(0, 1)} \right) \right\}^{2} \times \left\{ {\rm Im} \left( a_{1} \overline{z} \varphi_{(1, 0)} \right) \right\}^{2}\right\} \nonumber \\
& &\ \ \ \ \ \ \ \ +6 \left\{ {\rm Re} \left( a_{1} \overline{z} \varphi_{(1,0)} \right) \times {\rm Re} \left( a_{0} \overline{z} \varphi_{(0, 1)} \right) \times {\rm Im} \left( a_{1} \overline{z} \varphi_{(1,0)} \right) \times {\rm Im} \left( a_{0} \overline{z} \varphi_{(0, 1)} \right)\right\} \bigg] d \mu_{(p_{1}, p_{2}, p_{3})} \nonumber \\ 
\label{a1sqa0sq} 
& & \ \ \ \ \ \ \ \  \ \ \ \ \ \ \ \ \ \ \ \ \ \ \ \ \ \ \ \ \ \ \ \ \ \ \ \ \ \ \ \ \ \ \ \ \ \ \ \ \ \ \ \ \ \ \ \ \ \ \ \ \ \ \ \ \ \ \ \ \ \ \ \ \ \ \ \ \ \text{{\footnotesize (terms involving $a_{1}^{2} a_{0}^{2}$)}} \\ 
& & +\ O(a_{1}^{3}, a_{0}^{3}).\nonumber
\end{eqnarray} 

\subsubsection*{Case - 1 : $\mathbf{a} \in \mathbb{R}^{2}$ satisfying $P_{\mathbf{a}} \in \mathcal{Q}'$} 

\noindent 
We now estimate the right hand side of equation \eqref{cubiclyapcong} by considering the real vector $\mathbf{a}$ such that the cubic polynomial $P_{\mathbf{a}} \in \mathcal{Q}'$. We achieve this by estimating each of the integrals in \eqref{a1}, \eqref{a0},\eqref{a1a0}, \eqref{a1sq}, \eqref{a0sq},\eqref{a1sqa0}, \eqref{a1a0sq} and \eqref{a1sqa0sq} separately with $ a_{j} \equiv a_{j, \mathbb{R} }$ for $ j = 1, 0 $. We consider the values of every term in the equilibrium state where $p_{1} = p_{2} = p_{3} = \frac{1}{3}$ and when either of the $p_{i}$'s approach $1$. We infer that 
\begin{eqnarray}\label{realmatrix3} 
 \mathcal{L}_{\mu_{(p_{1}, p_{2}, p_{3})}} (P_\mathbf{a})  & = & -\log 3 - \int \log | \Phi_\mathbf{a} (z) | d \mu_{(p_{1}, p_{2}, p_{3})}  \nonumber\\
& &\begin{cases}
= & - \log 3 \quad \text{if} \quad p_{1} = p_{2} = p_{3}, \\ 
\to & - \log 3 + \frac{1}{2} a_{ 1} + \frac{1}{2} a_{ 0 } + \frac{3}{4} a_{ 1 } a_{ 0 } + \frac{1}{4} a_{ 1}^{2} + \frac{1}{2} a_{ 0}^{2}\\
& + \frac{15}{16} a_{ 1}^{2}a_{ 0 } + \frac{3}{2}a_{ 1}a_{ 0 }^{2}
 + 3a_{ 1}^{2} a_{ 0}^{2}  \quad\quad\quad\quad\quad\quad\quad \text{as} \quad p_{i} \uparrow 1.
\end{cases}
\end{eqnarray} 

\subsubsection*{Case - 2: $\mathbf{a} \in \mathbb{C}^{2}$ satisfying $P_{\mathbf{a}} \in \mathcal{Q}'$} 

\noindent 
In this case, we undertake the same calculations as in case - 1, however considering $\mathbf{a}$ to be a vector over the complex field. By considering terms upto coefficients of $a_{1}^{2}$ and $a_{0}^{2}$ in the right hand side of the equation \eqref{cubiclyapcong} and evaluating the integrals as written in  \eqref{a1}, \eqref{a0},\eqref{a1a0}, \eqref{a1sq}, \eqref{a0sq},\eqref{a1sqa0}, \eqref{a1a0sq} and \eqref{a1sqa0sq}, we infer that
\begin{eqnarray}\label{realmatrix3} 
 \mathcal{L}_{\mu_{(p_{1}, p_{2}, p_{3})}} (P_\mathbf{a})  & = & - \log 3 - \int \log | \Phi_\mathbf{a} (z) | d \mu_{(p_{1}, p_{2}, p_{3})}  \nonumber\\
& &\begin{cases}
= & - \log 3 \quad \text{if} \quad p_{1} = p_{2} = p_{3}, \\ 
\to & - \log 3 + \frac{1}{2} a_{ 1, \mathbb{R} } + \frac{1}{2} a_{ 0, \mathbb{R} } + \frac{3}{4} a_{ 1, \mathbb{R} } a_{ 0, \mathbb{R} } - \frac{3}{4} a_{ 1, \mathbb{I} } a_{ 0, \mathbb{I} } + \frac{1}{4} a_{ 1, \mathbb{R} }^{2} \\
& - \frac{1}{4} a_{ 1, \mathbb{I} }^{2} + \frac{1}{2} a_{ 0, \mathbb{R} }^{2} - \frac{1}{2} a_{ 0, \mathbb{I} }^{2} + \frac{15}{16} a_{ 1, \mathbb{R} }^{2}a_{ 0, \mathbb{R} } - \frac{15}{16} a_{ 1, \mathbb{I} }^{2}a_{ 0, \mathbb{R} } + \frac{3}{2}a_{ 1, \mathbb{R}}a_{ 0, \mathbb{R} }^{2} \\ 
& -  \frac{3}{2}a_{ 1, \mathbb{R} }a_{ 0, \mathbb{I} }^{2} -  \frac{15}{8}a_{ 1, \mathbb{R} }a_{ 1, \mathbb{I} }a_{ 0, \mathbb{I} } -3a_{ 1, \mathbb{I} }a_{ 0, \mathbb{R} }a_{ 0, \mathbb{I} } - 12a_{ 1, \mathbb{R} }a_{ 1, \mathbb{I} }a_{ 0, \mathbb{R} }a_{ 0, \mathbb{I} } \\ & + 3a_{ 1, \mathbb{R} }^{2} a_{ 0, \mathbb{R} }^{2} -3a_{ 1, \mathbb{R} }^{2} a_{ 0, \mathbb{I} }^{2} - 3a_{ 1, \mathbb{I} }^{2} a_{ 0, \mathbb{R} }^{2} + 3a_{ 1, \mathbb{I} }^{2} a_{ 0, \mathbb{I} }^{2} \quad \text{as} \quad p_{i} \uparrow 1.
\end{cases} \nonumber \\ 
& & 
\end{eqnarray} 

\section{Concluding Observations} 
\label{concl}

\noindent 
We conclude this paper with a few observations.
\begin{enumerate} 
\item \begin{itemize} 
\item In the case of the quadratic polynomial, the first and second order total derivative with respect to $c$ of the Lyapunov exponent when $c$ is real corresponds to the $(1, 1)$ entry of the appropriate matrix, when $c$ is complex, i.e., 
\[ \frac{d \mathcal{L}}{dc}\ \ =\ \ \frac{\partial \mathcal{L}}{\partial c_{\mathbb{R}}}\ \ \ \ \text{and}\ \ \ \ \frac{d^{2} \mathcal{L}}{dc^{2}}\ \ =\ \ \frac{\partial^{2} \mathcal{L}}{\partial c_{\mathbb{R}}^{2}}. \] 
Here, the left hand side corresponds to the case when $c$ is real while the right hand side corresponds to the case when $c$ is complex. 
\item Similarly, in the case of the cubic polynomial, we have  
\begin{eqnarray*} 
\frac{\partial \mathcal{L}}{\partial a_{1}}\ \ =\ \ \frac{\partial \mathcal{L}}{\partial a_{1, \mathbb{R}}}\bigg|_{a_{1, \mathbb{I}}\, =\, 0 }\bigg|_{a_{0, \mathbb{I}}\, =\, 0} & \text{and} & \frac{\partial \mathcal{L}}{\partial a_{0}}\ \ =\ \ \frac{\partial \mathcal{L}}{\partial a_{0, \mathbb{R}}}\bigg|_{a_{1, \mathbb{I}}\, =\, 0}\bigg|_{a_{0, \mathbb{I}}\, =\, 0}; \\ 
\frac{\partial^{2} \mathcal{L}}{\partial a_{1}^{2}}\ \ =\ \ \frac{\partial^{2} \mathcal{L}}{\partial a_{1, \mathbb{R}}^{2}}\bigg|_{a_{0, \mathbb{I}}\, =\, 0} & \text{and} & \frac{\partial^{2} \mathcal{L}}{\partial a_{0}^{2}}\ \ =\ \ \frac{\partial^{2} \mathcal{L}}{\partial a_{0, \mathbb{R}}^{2}}\bigg|_{a_{1, \mathbb{I}}\, =\, 0}. 
\end{eqnarray*} 
In all the equations above, the left hand side corresponds to the case when $a_{1}$ and $a_{0}$ are real while the right hand side corresponds to the case when $a_{1}$ and $a_{0}$ are complex. 
\end{itemize} 
\item \begin{itemize} 
\item When the parameter of the quadratic polynomial namely $c$ is complex, the second derivative of $\mathcal{L}$ with respect to the real part of $c$ and the second derivative of $\mathcal{L}$ with respect to the imaginary part of $c$ are the same in modulus; but different in sign, i.e., 
\[ \frac{\partial^{2} \mathcal{L}}{\partial c_{\mathbb{R}}^{2}}\ \ =\ \ - \frac{\partial^{2} \mathcal{L}}{\partial c_{\mathbb{I}}^{2}}. \]
\item The same property is true for the parameters of the cubic polynomial namely $a_{1}$ and $a_{0}$, when they are complex, \textit{i.e.}, 
\[ \frac{\partial^{2} \mathcal{L}}{\partial a_{1, \mathbb{R}}^{2}}\ \ =\ \ - \frac{\partial^{2} \mathcal{L}}{\partial a_{1, \mathbb{I}}^{2}}\bigg|_{a_{0, \mathbb{I}}\, =\, 0};\ \ \ \ \text{and}\ \ \ \ \frac{\partial^{2} \mathcal{L}}{\partial a_{0, \mathbb{R}}^{2}}\ \ =\ \ - \frac{\partial^{2} \mathcal{L}}{\partial a_{0, \mathbb{I}}^{2}}\bigg|_{a_{1, \mathbb{I}}\, =\, 0}. \]
\end{itemize} 
\item \begin{itemize} 
\item For a complex parameter $c$ under consideration, the determinant of the second differential operator (with respect to $c$) of the Lyapunov exponent in theorem \eqref{theorem1} turns out to be 
\begin{eqnarray}
{\rm det}\left[D_{c}^{2} \left( \mathcal{L}_{\mu_{(p_{1}, p_{2})}} (P_{c}) \right) \right] & &\begin{cases}
= 0 \quad & \text{if} \quad p_{1} = p_{2} = \frac{1}{2}, \\ 
\to -9 \quad & \text{as} \quad p_{1} \uparrow 1\ \text{or}\ p_{2} \uparrow 1. 
\end{cases}
\end{eqnarray} 
Define a function $g : \mathcal{J}_c \longrightarrow \mathbb{R}$ by
\[ (g \circ \Phi_{c} \circ \Psi) (\underline{x})\ \ :=\ \ 
\begin{cases}
\log p_{1}\ \quad \text{if}\ \quad x_{0} = 0, \\
\log p_{2}\ \quad \text{if}\ \quad x_{0} = 1;
\end{cases} \]
for $\underline{x} \in \Sigma_{2}^{+}$. Observe that $g$ is not only H\"{o}lder continuous, but in fact a locally constant function that only depends on the zeroth coordinate $x_{0}$ of the infinite vector $\underline{x} \in \Sigma_{2}^{+}$. Then we know from Walters \cite{Walter} that the measure $\mu_{(p_{1}, p_{2})}$ associated to the Bernoulli measure $\tilde{\mu}_{(p_{1}, p_{2})}$ is the equilibrium state assured by Denker and Urbanski in \cite{Denker} for the function $g$. Since $g$ is analytic in $p_{1}$, one may infer that the map $\Theta_{2} : p_{1} \longmapsto D_{c}^{2} \left[ \mathcal{L}_{\mu_{(p_{1}, p_{2})}} (P_{c}) \right]$ is analytic too. Observe that 
\begin{eqnarray*} 
\Theta_{2} (p_{1}) & \approx & \int \left[ {\rm Re}\left( c^{2} \overline{z} \varphi_{2} \right) - \frac{1}{2} \left( {\rm Re}\left(c \overline{z} \varphi_{1} \right) \right)^2 + \frac{1}{2}\left( {\rm Im}\left(c \overline{z} \varphi_{1} \right) \right)^2\right] d \mu_{(p_{1}, p_{2})}, 
\end{eqnarray*}
from equation \eqref{second|c|<1}. In other words, we have produced a non-zero analytic function that is approximately $-9$ when either $p_{1}$ or $p_{2}$ approaches $1$ by fixing $h (z) := {\rm Re}\left(c^{2} \overline{z} \varphi_{2} (z) \right) - \frac{1}{2} \left( {\rm Re}\left(c \overline{z} \varphi_{1} (z) \right) \right)^{2}+ \frac{1}{2} \left( {\rm Im }\left(c \overline{z} \varphi_{1} (z) \right) \right)^{2} $. Further, since we have $\Theta_{2}$ to be an analytic function due to corollary \eqref{Holder dimension}, we have proved that the Lyapunov exponent of the quadratic polynomial map $P_{c} \in \mathcal{Q}$ with $c \in \mathbb{R}$ can be zero or vary linearly in $c$ for only atmost finitely many measures $\mu_{(p_{1}, p_{2})}$ associated to the Bernoulli measure $\tilde{\mu}_{(p_{1}, p_{2})}$. 
\item Using same argument as in the case of the quadratic polynomial with a complex parameter, we conclude that by taking a complex vector $\mathbf{a}$ in theorem \eqref{theorem2}, we have 
\[ {\rm det} \left[ D^{2}_{\mathbf{a}} \left( - \int_{\mathcal{J}_{\mathbf{a}}} \log \left| P_{\mathbf{a}}' \right| d\mu_{(p_{1}, p_{2}, p_{3})} \right) \right]\ \ =\ \ 0, \]
only for finitely many $0 < p_{i} < 1$, for $1 \le i \le 3$. Here $D_{\mathbf{a}}$ is the differential operator with respect to $\mathbf{a} = (a_{1}, a_{0})$. 
\end{itemize} 
\end{enumerate}
\medskip 

\noindent 
The results due to Manning as in \cite{Manning} and Sridharan as in \cite{Sridharan} are now easily obtained as corollaries to our computations. We conclude the paper by writing these results.
\medskip 

\noindent 
\begin{corollary}[\cite{Manning}] 
The Lyapunov exponent of a hyperbolic, monic, centered, quadratic polynomial map $P_{c} (z) = z^{2} + c$ restricted on its Julia set $\mathcal{J}_{c}$ with respect to the Lyubich's measure $\mu$ is a constant, 
\[ \int_{\mathcal{J}_{c}} \log |P'_{c}| d \mu\ \ =\ \ \log 2. \] 
\end{corollary}
\medskip 

\noindent 
\begin{corollary}[\cite{Sridharan}]
Consider the hyperbolic, monic, centered, quadratic polynomial map $P_{c} (z) = z^{2} + c$ for $c \in [0, \frac{1}{4})$. Then, the second derivative (with respect to $c$) of the Lyapunov exponent  with respect to the measure $\mu_{(p_{1}, p_{2})}$ associated to the Bernoulli measure $\tilde{\mu}_{(p_{1}, p_{2})}$ can be zero for only finitely many $0 < p_{1} < 1$.
\end{corollary}

\bigskip 

\noindent 
{\sc Sridharan, Shrihari} \\
Indian Institute of Science Education and Research Thiruvananthapuram (IISER-TVM) \\
Maruthamala P.O., Vithura, Thiruvananthapuram, INDIA. PIN 695 551. \\ 
{\tt shrihari@iisertvm.ac.in}  
\bigskip 

\noindent 
{\sc Tiwari, Atma Ram} \\ 
Indian Institute of Science Education and Research Thiruvananthapuram (IISER-TVM) \\
Maruthamala P.O., Vithura, Thiruvananthapuram, INDIA. PIN 695 551. \\ 
{\tt artiwari15@iisertvm.ac.in}
\end{document}